\documentclass[11pt]{article}
\usepackage[latin1]{inputenc}
\usepackage{epsfig}
\usepackage{color}
\usepackage[british,english]{babel}
\usepackage{amsthm}
\usepackage{amsmath}
\usepackage{amsfonts}
\usepackage{amssymb}
\usepackage{graphicx}
\newtheorem{corollary}{Corollary}[section]

\newtheorem{remark}[corollary]{Remark}
\newtheorem{theorem}[corollary]{Theorem}
\newfont{\sBlackboard}{msbm10 scaled 900}

\newcommand{\mylabel}[1]{\label{#1}
            \ifx\undefined\stillediting
            \else \fbox{$#1$}\fi }
\newcommand{\BE}{\begin{equation}}

\newcommand{\EEQ}{\end{equation}}
\newcommand{\rfb}[1]{\mbox{\rm
   (\ref{#1})}\ifx\undefined\stillediting\else:\fbox{$#1$}\fi}

\newfont{\Blackboard}{msbm10 scaled 1200}

\newfont{\roma}{cmr10 scaled 1200}

\def\CC{\rm \hbox{C\kern-.56em\raise.4ex
         \hbox{$\scriptscriptstyle |$}\kern+0.5 em }}

\def\Frac{\displaystyle\frac}

\def\n{|\kern -.05cm{|}\kern -.05cm{|}}

\def\R{{\bf \hbox{\sc I\hskip -2pt R}}} 

\def\r{\rho}

\newcommand{\mm}    {{\hbox{\hskip 0.5pt}}}

\newcommand{\bluff} {{\hbox{\raise 15pt \hbox{\mm}}}}
\makeatletter
\def\section{\@startsection {section}{1}{\z@}{-3.5ex plus -1ex minus
    -.2ex}{2.3ex plus .2ex}{\large\bf}}
\makeatother
\def\be{\begin{equation}}
\def\ee{\end{equation}}

\date{ }
\begin{document}
\thispagestyle{empty}
\title{\bf   The universal bound property for a class of second order ODEs}\maketitle

\author{ \center  Mama ABDELLI (1,2)\\
1.Université Djillali Liabés, Laboratoire de Mathématique, \\B.P.89. Sidi Bel Abbés 22000,
Algeria.\\
2. Mustapha Stambouli University, Mascara, (Algeria) \\
abdelli$_{-}$mama@yahoo.fr\\}
\medskip\author{ \center  Alain HARAUX \\ Sorbonne Université, Université Paris-Diderot SPC, CNRS, INRIA, \\
Laboratoire Jacques-Louis Lions,  LJLL, F-75005,
Paris, France.\\ haraux@ann.jussieu.fr\\}

\vskip20pt

 \renewcommand{\abstractname} {\bf Abstract}
\begin{abstract} We consider  the scalar second order ODE  $ u'' +
|u'|^{\alpha}u' + |u|^\beta u=0$, where $\alpha, \beta$ are two positive numbers and the non-linear semi-group $S(t)$ generated on $\R^2$ by the system in $(u, u')$. We prove that $S(t)\R^2$ is bounded for all $t>0$ whenever $0< \alpha<\beta$ and moreover there is a constant $C$ independent of the initial data such that
    $$ \forall t>0,\,\,\,\ u'(t)^{2} +|u(t)|^{\beta +2}\leq C \max \{t ^{- \frac {2}{\alpha}}, \,\, t^{-\frac {(\alpha+1)(\beta + 2)}{\beta -\alpha}}\}.$$

\end{abstract}
\bigskip\noindent
 {\small \bf AMS classification numbers:} 34A34, 34C10, 34D05, 34E99

\newpage
\section {Introduction}
 It is well known that a certain number of dynamical systems $S(t)$ defined on a Banach space $X$ have the property of universal boundedness for all $t>0$, in the sense that 
 \begin{equation} \forall t>0, \quad S(t) X \hbox { is a bounded subset of } X.   \end{equation}  As a simple example we can consider the first order scalar ODE 
 \begin{equation}  u' + \delta |u|^{\r} u = 0 \end{equation}  with $\delta, \r $ positive. Indeed for any initial state  $u(0) = u_0$ we have 
 $$ |u(t)| \le \frac {C}{t^{1/\r}} $$ Where $C= {(\r\delta)}^{-1/\r} $ is independent of $u_0$. This property extends  classically to some classes of nonlinear parabolic PDEs, for instance the semilinear parabolic equation $$ u_t -\Delta u + \delta |u|^{\r} u = 0 $$  with either Dirichlet or Neumann homogeneous boundary conditions, the result follows at once from the maximum principle. For a more elaborate quasilinear case,  cf.  \cite{Simon}.

 It is natural to ask  whether second order ODEs with superlinear dampings such as  \begin{equation}  u''+ \omega^2u + \delta |u'|^{\r} u' = 0 \end{equation} have the same property, however  the equation \begin{equation}  u''+ u +  |u'|u' = 0 \end{equation} has an explicit solution defined for all $t$ on the line, more precisely $ u(t) = \frac{1}{4} t^2 -\frac{1}{2}$ is a solution for $t\le 0$ which extends uniquely for $t\ge0$ and has an unbounded range. In such a case, due to the autonomous character of the equation, the entire range of the unbounded solution is contained in $S(t)\R^2$ for all $t>0$  and universal boundedness fails.   In this paper we consider the scalar second order ODE
 \begin{equation}\label{1}
 u'' + |u'|^{\alpha}u' + |u|^\beta u=0,
 \end{equation} for which, after the limited elementary approach of  \cite{har0}, Philippe Souplet \cite{Souplet} gave a definitive negative answer at least when $\alpha\ge \beta\ge 0$. Strangely enough, although this equation has been studied very carefully in  \cite{har1} which several extensions in \cite{AH, A^2H}, it seems that nobody questionned the possibility of universal boundedness  when $0< \alpha < \beta$.

 \section {Universal bounds  for equation (\ref{1}) when $0< \alpha < \beta$. }
 The fact that the semi-group $S(t)$ generated by equation (\ref{1}) is universally bounded  when $0< \alpha < \beta$ for $t>0$ would actually follow from a careful inspection of the proof of Theorem 3, (a) from \cite{Souplet}, since this proof allows to see that the life-time of the solutions to the backward equation tends to $0$ when the size of the initial state tends to infinity. We do not discuss here well-posedness of the initial-value problem for equation (\ref{1}) which is completely standard. For any solution $u$ on $\R$,  the energy defined by 
the  formula
\begin{equation}\label{2}
E(t) = \Frac{1}{2}u'(t)^{2} + \Frac{1}{\beta +2}|u(t)|^{\beta +2}.
\end{equation}
is a good way of measuring the size of the solution at time $t$. 
By multiplying equation (\ref{1}) by $u'$, we find \begin{equation}\label{A3}
\Frac{d}{dt}E(t) = -|u'(t)|^{\alpha+2}\leq 0.
\end{equation}
hence the energy of any solution is non-increasing. Moreover, universal boundedness of the dynamical system generated by equation (\ref{1}) is equivalent to the existence of a  bound of the energy independent of the solution. We shall give two different bounds when $t$ tends to zero depending on the relative positions of $\alpha$ and $\frac{\beta}{\beta+2}$. The first  result of this paper is 
\begin{theorem}\label{Th1}
Assuming  $0< \alpha \le  \frac{\beta}{\beta+2},$
          there is a constant $C$ independent of
          $E(0)$ such that
     $$
     \forall t\leq 1,\,\,\,\ E(t)\leq C t^{-\frac{2}{\alpha}
     }.
     $$
\end{theorem}
\begin{proof}
We consider the perturbed energy function
\begin{equation}\label{4}
F_\varepsilon (t) = E(t) + \varepsilon  |u|^{\frac{\alpha}{2}(\beta+2)}uu',
\end{equation}
where 
$\varepsilon > 0.$ To prove that $F_\varepsilon (t) \sim E(t) $, we use (\ref{2})
and Young's inequality, we have
\begin{equation}
\begin{split}\label{A5}
\vert |u|^{\frac{\alpha}{2}(\beta+2)}u u'\vert &\leq
 \frac{1}{2}|u'|^{2}+\frac{1}{2}|u|^{\alpha(\beta+2)+2}\\&
 \leq \frac{\beta+2}{2}\Big(\frac{1}{2}|u'|^{2}+ \frac{1}{\beta+2}|u|^{\alpha(\beta+2)+2}\Big)\\&\leq \frac{\beta+2}{2} E(t) +\frac{1}{2}  
\end{split}
\end{equation} since the condition $\alpha \le  \frac{\beta}{\beta+2}$ yields $ \alpha(\beta+2)+2\le {\beta+2}$, therefore
$$ \forall u\in \R,\quad |u|^{\alpha(\beta+2)+2}\le |u|^{\beta+2} + 1 $$
Then, by using (\ref{A5}), we obtain from (\ref{4})
$$
(1-M\varepsilon)E(t)- \varepsilon/2\leq F_\varepsilon(t) \leq (1+M\varepsilon)E(t) + \varepsilon/2.
$$ with $M := \frac{\beta+2}{2}.$
Taking  $\varepsilon\leq \varepsilon_0=\Frac{1}{2M}$, we deduce
\begin{equation}\label{6}
\forall t\geq 0,\,\,\, \Frac{1}{2}E(t)- \Frac{1}{4}\leq F_\varepsilon(t)\leq 2 E(t)+\Frac{1}{4} .
\end{equation}
Using (\ref{1}) and (\ref{A3}), we obtain
{\small\begin{equation}\label{33a}
\begin{split}
F'_\varepsilon(t) &= -|u'|^{\alpha+2} + \varepsilon
(\frac{\alpha}{2}(\beta+2)+1)|u|^{\frac{\alpha}{2}(\beta+2)}u'^{2}+ \varepsilon
|u|^{\frac{\alpha}{2}(\beta+2)}u(-|u'|^{\alpha}u'-|u|^{\beta}u),\\ 
&= -|u'|^{\alpha+2} -\varepsilon  |u|^{(\beta+2)(\frac{\alpha}{2}+1)} +\varepsilon
(\frac{\alpha}{2}(\beta+2)+1)|u|^{\frac{\alpha}{2}(\beta+2)}u'^{2} -\varepsilon
|u|^{\frac{\alpha}{2}(\beta+2)} u |u'|^\alpha u'.
\end{split}
\end{equation}}
By using Young's inequality, with the conjugate exponents
$\frac{\alpha+2}{2}$ and $\frac{\alpha+2}{\alpha}$, we get
$$
u'^2|u|^{\frac{\alpha}{2}(\beta+2)} \leq \delta|u|^{(\beta+2)(\frac{\alpha}{2}+1)}
+ C(\delta)|u'|^{\alpha+2},
$$
and taking  $\delta$ small enough, we have for some $P>0$
\begin{equation}\label{123}
\varepsilon (\frac{\alpha}{2}(\beta+2)+1)u'^2|u|^{\frac{\alpha}{2}(\beta+2)} \leq
\frac{\varepsilon }{4}|u|^{(\beta+2)(\frac{\alpha}{2}+1)}
+ P\varepsilon|u'|^{\alpha+2}.
\end{equation}
Using (\ref{123}), we have from (\ref{33a}) that
\begin{equation}\label{234}
F_\varepsilon'(t)\leq (-1+P\varepsilon)|u'|^{\alpha+2}
-\Frac{3\varepsilon}{4}|u|^{(\beta+2)(\frac{\alpha}{2}+1)}-\varepsilon|u|^{\frac{\alpha}{2}(\beta+2)}u|u'|^{\alpha}u'.
\end{equation}
Applying Young's inequality, with the conjugate exponents
$\frac{\alpha+2}{\alpha+1}$ and $\frac{1}{\alpha+2}$, we have
$$
-|u|^{\frac{\alpha}{2}(\beta+2)}u|u'|^{\alpha}u'\leq \delta|u|^{\frac{\alpha+2}{2}(2+\alpha(\beta+2))}
+C'(\delta)|u'|^{\alpha+2}.
$$
If $\alpha \le \frac{\beta}{\beta+2}$, then clearly 
$$
\frac{\alpha+2}{2}(2+\alpha(\beta+2)) \le(\beta+2)(\frac{\alpha}{2}+1),$$ 
then we have 
$$
\forall u \in \mathbb{R},\,\,\, |u|^{\frac{\alpha+2}{2}(2+\alpha(\beta+2))}\leq |u|^{(\beta+2)(\frac{\alpha}{2}+1)}+1,
$$
and taking $\delta$ small enough
$$
-\varepsilon|u|^{\frac{\alpha}{2}(\beta+2)}u|u'|^{\alpha}u'\leq \Frac{\varepsilon }{4}|u|^{(\beta+2)(\frac{\alpha}{2}+1)}
+P'\varepsilon|u'|^{\alpha+2}+C.
$$
By replacing it  in (\ref{234}), we have
$$
F'_\varepsilon(t)\leq(-1+Q\varepsilon)|u'|^{\alpha+2}-\Frac{\varepsilon}{2}|u|^{(\beta+2)(\frac{\alpha}{2}+1)}+C,
$$
where $Q=P+P'$. By choosing $\varepsilon$ small, we get
{\small\begin{equation}\label{A0}
 \begin{split}
F'_\varepsilon(t)&\leq-\frac{\varepsilon}{2}\Big(|u'|^{\alpha+2}
+|u|^{(\beta+2)(\frac{\alpha}{2}+1)}\Big)+C\\&\leq -\frac{\varepsilon}{2}
\Big((u'^{2})^{\frac{\alpha+2}{2}}
+(|u|^{\beta+2})^{\frac{\alpha+2}{2}}\Big)+C.
\end{split}
\end{equation}}
and also employing the inequality 
$$
1+x^{1+\mu}\leq (1+x)^{1+\mu}\leq 2^{\mu}(1+x^{1+\mu})
$$ for an arbitrary $\mu>0$, we obtain with $x=\frac{|u|^{\beta+2}}{u'^2}$ and $\mu =\frac{\alpha}{2}$
$$
(u'^{2})^{\frac{\alpha}{2}+1}+(|u|^{\beta+2})^{\frac{\alpha}{2}+1}\leq
 \Big(u'^{2}+|u|^{\beta+2}\Big)^{\frac{\alpha}{2}+1}
\leq 2^{\frac{\alpha}{2}}\Big((u'^{2})^{\frac{\alpha}{2}+1}+(|u|^{\beta+2})^{\frac{\alpha}{2}+1}\Big)
$$
Hence (\ref{A0}) becomes
{\small\begin{equation}\label{A00}
 \begin{split}
F'_\varepsilon(t)& \leq -\frac{\varepsilon}{2^{\frac{\alpha}{2}+1}}
\Big(u'^{2}
+|u|^{\beta+2}\Big)^{\frac{\alpha}{2}+1}+C_1,\,\,\,\,\,\,\, \forall t\leq 1.
\end{split}
\end{equation}}
We now use (\ref{6}) to deduce from (\ref{A00}) that
\begin{equation}\label{S1}
F'_\varepsilon(t) \leq -\rho' E^{\frac{\alpha}{2}+1}(t)+C_2\leq-\rho F^{\frac{\alpha}{2}+1}_\varepsilon(t)+C_3,\,\,\,\,\,\, \forall t\leq 1.
\end{equation}
 By Ghidaglia's inequality (cf. e.g. \cite{Temam}, III, Lemma 5.1) we infer \begin{equation}\forall t\leq 1,\quad 
F_\varepsilon(t) \leq \left(\frac{C_3}{\r}\right)^{\frac {2}{\alpha +2}} + \left(\frac{2t}{\r\alpha}\right)^{- \frac {2}{\alpha}}
\end{equation}  and then \begin{equation}\forall t\leq 1,\quad 
E(t) \leq 2\left(\frac{C_3}{\r}\right)^{\frac {2}{\alpha +2}} + 2\left(\frac{2t}{\r\alpha}\right)^{- \frac {2}{\alpha}} + 1
\end{equation}the result follows easily with $C= 2 \left(\frac{C_3}{\r}\right)^{\frac {2}{\alpha +2}} + 2\left(\frac{2}{\r\alpha}\right)^{- \frac {2}{\alpha}} +1 $ since $$ 2\left(\frac{C_3}{\r}\right)^{\frac {2}{\alpha +2}} +1 \le \left [2\left(\frac{C_3}{\r}\right)^{\frac {2}{\alpha +2}} +1)\right]  t^{- \frac {2}{\alpha}}$$ for all $t\in (0, 1]$
 .
\end{proof}

\begin{theorem}\label{Th2}
Assuming  $\beta>0$ and $\frac{\beta}{\beta+2}\le \alpha < \beta ,$
          there is a constant $C$ independent of
          $E(0)$ such that
     $$
     \forall t\leq 1,\,\,\,\ E(t)\leq C t^{-\frac {(\alpha+1)(\beta + 2)}{\beta -\alpha}
     }.
     $$
\end{theorem}
\begin{proof} Since it is quite similar to the proof of Theorem 2.1 we only sketch out the crucial steps. We consider now the perturbed energy 
\begin{equation}\label{4bis}
G_\varepsilon (t) = E(t) + \varepsilon  |u|^{\gamma}uu',
\end{equation}
where $ \gamma: = \frac{\beta-\alpha}{\alpha +1}$ and  $\varepsilon > 0.$ We again use Young's inequality which provides\begin{equation}
\label{A6}
\vert |u|^{\gamma} u u'\vert \leq
 \frac{1}{2}|u'|^{2}+\frac{1}{2}|u|^{2\gamma+2}
\end{equation} and we observe that the condition $\frac{\beta}{\beta+2}\le \alpha$ now implies $ 2\gamma+2 \le \beta+2$. 
Then, by using (\ref{A6}), we obtain from (\ref{4bis}) that for $\varepsilon $ small enough \begin{equation}
\forall t\geq 0,\,\,\, \Frac{1}{2}E(t)- \Frac{1}{4}\leq G_\varepsilon(t)\leq 2 E(t)+\Frac{1}{4} .
\end{equation}

\begin{equation}\label{33a}
\begin{split}
G'_\varepsilon(t) &= -|u'|^{\alpha+2} + \varepsilon
(\gamma +1)|u|^{\gamma}u'^{2}+ \varepsilon
|u|^{\gamma}u(-|u'|^{\alpha}u'-|u|^{\beta}u),\\ 
&= -|u'|^{\alpha+2} -\varepsilon  |u|^{\beta+ \gamma+2} +\varepsilon [(\gamma +1)|u|^{\gamma}u'^{2} -|u|^{\gamma}u |u'|^{\alpha}u'].
\end{split}
\end{equation}
The two first terms in the last line are perfect for our purpose. It turns out that they control the two last terms up to some constants. Indeed we have  

$$ |u|^{\gamma}u'^{2} \le \delta |u|^{\gamma (\frac {\alpha+2}{\alpha})} + C(\delta )|u'|^{\alpha+2} $$ As observed previously we have 
$2\gamma \le \beta $, hence $ \frac{2\gamma}{\alpha}\le \beta+2 $ and finally $\gamma (\frac {\alpha+2}{\alpha}) \le \beta+\gamma+ 2$. Therefore 
$$ \forall u \in \mathbb{R},\,\,\,|u|^{\gamma (\frac {\alpha+2}{\alpha})} \le |u|^{\beta+ \gamma+2} +1 $$  On the other hand we have 
$$ |u|^{\gamma+1} |u'|^{\alpha+1} \le |u'|^{\alpha+2} + |u|^{(\gamma+1)(\alpha+2)}$$ and it turns out that $ (\gamma+1)(\alpha+2) = \beta+ \gamma + 2.$ After a few manipulations we end up with the inequality 

$$ G'_\varepsilon(t)\le  - \frac{1}{2}|u'|^{\alpha+2}   -\frac{1}{\varepsilon } |u|^{\beta+ \gamma+2} +C$$ Taking account the fact that  $ \frac{2\gamma}{\alpha}\le \beta+2 $ implies 
$ \frac {\alpha+2}{2}\ge \frac {\beta + \gamma+2}{\beta+2}$, we deduce for some positive constants $\mu, C'$
$$ G'_\varepsilon(t)\le  - \mu E^{1+\frac{\gamma}{\beta+2} } +C' $$ and the conclusion follows by the same argument as in the proof of Theorem 2.1. 

\end{proof}

 \section {A sharp universal bound for equation (\ref{1}) when $0< \alpha < \beta$. } Our final result is in fact optimal and has the same appearance in both regions considered in the last section, although the significance is different according to the relative position of $\alpha$ and $\frac{\beta}{ \beta+2}.$ In the RHS of the inequality below, the smaller exponent is relative to the region $\{t\ge1\}$ (decay rate) and the larger exponent is relative to the region $\{t\le1\}$ (indicating the growth of the universal bound as $t$ tends to $0$. The roles of the exponents are exchanged when crossing the value  $\alpha_0=\frac{\beta}{ \beta+2}.$ 

\begin{theorem}\label{Th3}
Assuming  $0< \alpha < \beta$   there is a constant $C$ independent of
          $E(0)$ such that
    \begin{equation}\label{final} \forall t>0,\,\,\,\ E(t)\leq C \max \{t ^{- \frac {2}{\alpha}}, \,\, t^{-\frac {(\alpha+1)(\beta + 2)}{\beta -\alpha}}\} \end{equation}
\end{theorem}
\begin{proof} Let us consider first the case  $0< \alpha \le \alpha_0$. By Theorem 2.1 we have  for a certain constant $C_1$, 
$$\forall t\in (0, 1], \quad E(t)\leq C_1 t^{-\frac{2}{\alpha}}$$ In particular we have for any solution $ E(1) \le C_1.$  As a consequence of \cite{har1}, Theorem 1.1, ii)  we then have 
$$\forall t\ge 2, \quad E(t)\leq K(C_1)( t-1)^{-\frac{(\alpha+1)(\beta + 2)}{\beta -\alpha}}\leq C_2 t^{-\frac{(\alpha+1)(\beta + 2)}{\beta -\alpha}}$$ Now for $t\ge 1$ we have $t^{-\frac{2}{\alpha}}\le t^{-\frac{(\alpha+1)(\beta + 2)}{\beta -\alpha}}$. In particular since E(t) is non-increasing we have  $$\forall t\in [1,  2], \quad E(t) \leq C_1 (t/2)^{-\frac{2}{\alpha}}\le C_3 t^{-\frac{(\alpha+1)(\beta + 2)}{\beta -\alpha}}$$ Introducing  $C = \max\{C_1,C_2, C_3 \}$ we finally obtain \eqref{final}.  The case $ \alpha_0 \le \alpha<\beta $ is similar, using Theorem 2.2 and \cite{har1}, Theorem 1.1, i). 

 \end{proof}
 
  \begin{remark} The estimate \eqref{final} is optimal. For instance for $ \alpha_0 \le \alpha<\beta $ the decay estimate is sharp for all nontrivial solutions, while the growth for $t$ tending to $0$ is sharp for some special solutions of the backward equation found by Souplet \cite{Souplet}. On the other side, when $0< \alpha \le \alpha_0$, all non-trivial solutions of the backward equation blow-up exactly at the rate allowed by the estimate, while the decay estimate is sharp for the "slow-decaying" solutions, cf. \cite{har1}.\end{remark}
 
 \begin{remark} It is clear that Theorem \ref{Th3} can be extended under relaxed assumptions on the non-linearities, and for vector equations such as those considered in \cite{A^2H}. This will be the object of future works. \end{remark}


\begin{thebibliography}{10}

\bibitem{AH}  M. Abdelli \& A. Haraux, {\em Global behavior of the solutions to a class of nonlinear, singular second order ODE}, Nonlinear Anal. {\bf 96} (2014), 18-37. 

\bibitem{A^2H}  M. Abdelli, M. Anguiano \& A. Haraux,  {\em  Existence, uniqueness and global behavior of the solutions to some nonlinear vector equations in a finite dimensional Hilbert space},  Nonlinear Anal. {\bf 161} (2017), 157-181.

\bibitem{har0}  A. Haraux, {\em Remarks on the wave equation with a nonlinear term with respect to the velocity},
Portugal. Math. {\bf 49} (1992), no. 4, 447-454. 

\bibitem{har1}  A. Haraux, {\em Sharp decay estimates of the solutions to a class
of nonlinear second order ODE's\/,} Analysis and Applications,
{\bf 9} (2011), 49-69.

\bibitem{Simon}  J. Simon, {\em Quelques propriétés de solutions d'équations et d'inéquations d'évolution paraboliques non linéaires},
Ann. Scuola Norm. Sup. Pisa Cl. Sci. (4) 2 (1975), no. 4, 585-609. 

\bibitem{Souplet} 
P. Souplet {\em Critical exponents, special large-time behavior and oscillatory blow-up in nonlinear ODE's}, 
Differential Integral Equations {\bf 11} (1998), no. 1, 147-167. 

\bibitem{Temam}  R. Temam,  {\em Infinite-dimensional dynamical systems in mechanics and physics. } Applied Mathematical Sciences, 68. Springer-Verlag, New York, 1988. xvi+500 pp. ISBN: 0-387-96638-2 
   
\end{thebibliography}
\end{document}